\documentclass[12pt]{article}
\usepackage{zbarskystyle}
\usepackage[colorlinks=true,citecolor=black,linkcolor=black,urlcolor=blue]{hyperref}
\title{\bf The Maximum Number of Subset Divisors of a Given Size}
\author{Samuel Zbarsky\\
\small Carnegie Mellon University\\
\small\tt sa\_zbarsky@yahoo.com\\
}

\date{\small Mathematics Subject Classifications: 05A15, 05D05}

\begin{document}
\maketitle

\section{Abstract}
If $s$ is a positive integer and $A$ is a set of positive integers, we say that $B$ is an $s$-divisor of $A$ if $\sum_{b\in B} b\mid s\sum_{a\in A} a$. We study the maximal number of $k$-subsets of an $n$-element set that can be $s$-divisors. We provide a counterexample to a conjecture of Huynh that for $s=1$, the answer is $\binom{n-1}{k}$ with only finitely many exceptions, but prove that adding a necessary condition makes this true. Moreover, we show that under a similar condition, the answer is $\binom{n-1}{k}$ with only finitely many exceptions for each $s$.

\section{Introduction}
If $X$ is a set of positive integers, let $\sum X$ denote $\sum_{x\in X}x$.
Let $A$ be a finite subset of the positive integers. The elements of $A$ are $a_1<a_2<\cdots<a_n$ and let $B$ be a subset of $A$.
We say that $B$ is a \emph{divisor} of $A$ if $\sum B\mid\sum A$. 
We define $d_k(A)$ to be the number of $k$-subset divisors of $A$ and let $d(k,n)$ be the maximum value of $d_k(A)$ over all sets $A$ of $n$ positive integers.

Similarly, for $s\ge 1$ a positive integer, we say that $B$ is an \emph{s-divisor} of $A$ if $\sum B\mid s\sum A$.
We define $d^s_k(A)$ to be the number of $k$-subset $s$-divisors of $A$ and let $d^s(k,n)$ be the maximum value of $d^s_k(A)$ over all sets $A$ of $n$ positive integers.

Note that the concepts of divisor and 1-divisor coincide. Also, if $B$ is a divisor of $A$, then $B$ is an $s$-divisor of $A$ for all $s$, so $d^s_k(A)\ge d_k(A)$ and $d^s(k,n)\ge d(k,n)$

Huynh~\cite{huynh14} notes that for any values of $a_1,\ldots,a_{n-1}$, we can pick such an $a_n$ that any $k$-subset of $\{a_1,\ldots,a_{n-1}\}$ will be an $A$-divisor.
Therefore $d(k,n)\ge \binom{n-1}{k}$ for all $1\le k\le n$.
This motivates the definition that $A$ is a \emph{k-anti-pencil} if the set of $k$-subset divisors of $A$ is $\binom{A\wo\{a_n\}}{k}$.
We similarly define $A$ to be a $(k,s)$-\emph{anti-pencil} if the set of $k$-subset $s$-divisors of $A$ is $\binom{A\wo\{a_n\}}{k}$.

Huynh~\cite{huynh14} also formulates the following conjecture (Conjecture 22).
\begin{conjecture}\label{huynhconj}
For all but finitely many values of $k$ and $n$, $d(k,n)=\binom{n-1}{k}$.
\end{conjecture}

In this paper, we provide infinite families of counterexamples, but prove that, with the exception of these families, the conjecture is true. This gives us the following modified form.

\begin{conjecture}\label{huynhconjmod}
For all but finitely many integer pairs $(k,n)$ with $1<k<n$, $d(k,n)=\binom{n-1}{k}$.
\end{conjecture}

For convenience, we now rescale, dividing every element of $A$ by $\sum A$, so that now the elements of $A$ are positive rational numbers and $\sum A=1$. Under this rescaling, $B\subseteq A$ is a divisor of $A$ if and only if $\sum B=\frac{1}{m}$ for some positive integer $m$ and $B$ is an $s$-divisor of $A$ if and only if $\sum B=\frac{s}{m}$ for some positive integer $m$. Clearly, the values of $d(k,n)$ and $d^s(k,n)$ do not change.

The $k<n$ condition in Conjecture~\ref{huynhconjmod} is necessary since it is easy to see that $d(n,n)=1>\binom{n-1}{n}$. Also, if
\[
A=\left\{\half,\frac{1}{4},\ldots,\frac{1}{2^{n-2}},\frac{1}{3(2^{n-1})},\frac{1}{3(2^{n-2})}\right\}
\]
then $\sum A=1$, so $d_1(A)=n$ and $d(1,n)\ge n>\binom{n-1}{1}$. Therefore the $1<k$ condition is necessary.

However, we prove that these families cover all but finitely many exceptions.
\begin{theorem}\label{s1case}
For all but finitely many pairs $(k,n)$, if $1<k<n$, $|A|=n$, and $d_k(n)\ge\binom{n-1}{k}$, then $A$ is a $k$-anti-pencil.
\end{theorem}

Note that this immediately implies Conjecture~\ref{huynhconjmod}.

If we are interested in $s$-divisors, we get another family of exceptions. If $s\ge 2$, $a_n=\frac{1}{s+1}$ and $a_{n-1}=\frac{2}{s+2}$, then $d^s_{n-1}(A)\ge 2$, so $d^s(n-1,n)\ge 2>\binom{n-1}{n-1}$. However, we prove that these cover all but finitely many exceptions.

\begin{theorem}\label{generalcase}
Fix $s\ge 1$. For all but finitely many pairs $(k,n)$ (with the number of these pairs depending on $s$), if $1<k<n-1$, $|A|=n$, and $d^s_k(n)\ge\binom{n-1}{k}$, then $A$ is a $(k,s)$-anti-pencil.
\end{theorem}

Note that this immediately implies the following corollary.
\begin{corollary}\label{generalcasecorr}
Fix $s\ge 1$. Then $d^s(k,n)=\binom{n-1}{k}$ for all but finitely many pairs $(k,n)$ with $1<k<n-1$ (with the number of these pairs depending on $s$).
\end{corollary}

We will prove Theorem~\ref{generalcase}. In the $s=1$ case, where $k=n-1$, if $i\le n-1$, then $\sum(A\wo\{a_i\})>\half$, so $A\wo\{a_i\}$ is not a divisor of $A$. This, together with the $s=1$ case of Theorem~\ref{generalcase}, gives us Theorem~\ref{s1case}.

\section{Lemmas} 

Take a $d$-dimensional lattice cube with $n$ lattice points per edge. Define a poset on the lattice points by $(x_1,\ldots,x_d)\le(y_1,\ldots,y_d)$ if $x_i\le y_i$ for all $i$.
\begin{lemma}
The largest antichain in this poset has at most $(n+d-2)^{d-1}\sqrt{\frac{2}{d}}$ elements. 
\end{lemma}

\begin{proof}
First, we need some definitions.

The \emph{width} of a poset is the size of its largest antichain.
If $P$ is a finite poset, we say that $P$ is \emph{ranked} if there exists a function $\rho:P\to \Z$ satisfying $\rho(y)=\rho(x)+1$ if $y$ covers $x$ in $P$ (i.e.\ $y>x$, and there is no $z\in P$ with $y>z>x$). If $\rho(x)=i$, then $x$ is said to have \emph{rank i}.
Let $P_i$ denote the set of elements of $P$ of rank $i$. We say $P$ is \emph{rank-symmetric rank-unimodal} if there exists some $c\in\Z$ with $|P_i|\le|P_{i+1}|$ when $i<c$ and $|P_{2c-i}|=|P_i|$ for all $i\in \Z$.
A ranked poset $P$ is called \emph{strongly Sperner} if for any positive integer $s$, the largest subset of $P$ that has no $(s+1)$-chain is the union of the $s$ largest $P_i$.

Proctor, Saks, and Sturtevant~\cite{ProctorSturtevant1980} prove that the class of rank-symmetric rank-unimodal strongly Sperner posets is closed under products.

Since a linear ordering of length $n$ is rank-symmetric rank-unimodal strongly Sperner, so is a product of $d$ of them (the lattice cube).

Center the cube on the origin by translation in $\R^d$. Let $U$ be the set of elements whose coordinates sum to 0. Since the poset is rank-symmetric rank-unimodal strongly Sperner, its width is at most the size of $P_c$, which is $|U|$. 

For each $y=(y_1,\ldots,y_d)\in U$, let $S_y$ be the set of points $(x_1,\ldots,x_d)$ with $|x_i-y_i|<\half$ for $1\le i\le d-1$ (note that this does not include the last index) which lie on the hyperplane given by $x_1+\cdots+x_d=0$. If $y,z$ are distinct elements of $U$, then $S_y$ and $S_z$ are clearly disjoint. Also, the projection of $S_y$ onto the hyperplane given by $x_d=0$ is a unit $(d-1)$-dimensional hypercube, which has volume 1. Thus the volume of $S_y$ is $\sqrt{d}$ and the volume of $\bigcup_{y\in U}S_y$ is $|U|\sqrt{d}$.

On the other hand, if $(x_1,\ldots,x_d)\in S_y$, then $|x_i-y_i|<\half$ for $1\le i\le d-1$ and $|x_d-y_d|\le\sum_{i=1}^{d-1}|x_i-y_i|<\half(d-1)$. Thus $(x_1,\ldots,x_d)$ lies in the cube of edge length $(n-1)+(d-1)=n+d-2$ centered at the origin. Therefore $\bigcup_{y\in U}S_y$ lies in the intersection of a cube of edge length $n+d-2$ with a hyperplane through its center (the origin).

Ball~\cite{Ball1986} shows that the volume of the intersection of a unit hypercube of arbitrary dimension with a hyperplane through its center is at most $\sqrt{2}$. Therefore the volume of 
$\bigcup_{y\in U}S_y$ is at most $(n+d-2)^{d-1}\sqrt{2}$, so
\[
|U|\le (n+d-2)^{d-1}\sqrt{\frac{2}{d}}.
\]
\end{proof}

Let $X=\{x_1<\cdots<x_n\}$ be any set of positive integers. If $B,C\in \binom{X}{d}$, then we say that $B\le C$ if we can write $B=\{b_1,\ldots,b_d\}$ and $C=\{c_1,\ldots,c_d\}$ with $b_i\le c_i$ for all $1\le i\le d$. Whenever we compare subsets of $A$, we will be using this partial order.

\begin{lemma}\label{simplex}
Fix $d>1$. For $n$ sufficiently large, the width of the partial order defined above is less than $\frac{2}{\sqrt{d}}\frac{1}{n}\left|\binom{X}{d}\right|$.
\end{lemma}

\begin{proof}
Let $U$ be a maximum antichain of the partial order. Take the partial order of $X^d$, which coincides with the cube partial order. Let $U'=\{(y_1,\ldots,y_d)\in X^d\mid \{y_1,\ldots,y_d\}\in U\}$. Note that this means, in particular, that all elements of any $k$-tuple in $U'$ are distinct. If $(y_1,\ldots,y_d),(z_1,\ldots,z_d)\in U'$ with $(y_1,\ldots,y_d)<(z_1,\ldots,z_d)$, then we get that $\{y_i\}\le\{z_i\}$ and $\sum_{i=1}^d y_i< \sum_{i=1}^d z_i$, so $\{y_i\}\ne\{z_i\}$, so $\{y_i\}<\{z_i\}$, which is impossible. Thus $U'$ is an antichain of $X^d$ of size $d!|U|$ and 
\[
|U|\le\frac{1}{d!}(n+d-2)^{d-1}\sqrt{\frac{2}{d}}.
\]
Then $\left|\binom{X}{d}\right|=\binom{n}{d}$ gives us
\[
\frac{|U|}{|\binom{X}{d}|}\le \frac{(n+d-2)^{d-1}\sqrt{\frac{2}{d}}}{n(n-1)\dotsm(n-d+1)}.
\]
For sufficiently large $n$, 
$\left(\frac{n+d-2}{n-d+1}\right)^{d-1}<\sqrt{2}$, so $\frac{|U|}{|\binom{X}{d}|}<\frac{2}{\sqrt{d}}\frac{1}{n}$.
\end{proof}

Let $d(n)$ denote the number of divisors of $n$.

\begin{lemma} \label{divisors}
For any positive integer $k$, $d(n)=O(n^\frac{1}{k})$.
\end{lemma}

\begin{proof}
There are finitely many primes $p<2^k$, so there must be some constant $C$ such that for any $p<2^k$ and any positive integer $m$, $d(p^m)=m+1\le C(p^m)^\frac{1}{k}$.

For $p>2^k$, $d(p^m)=m+1\le 2^m\le (p^m)^\frac{1}{k}$. Thus if $n=\prod_{i=1}^j p_i^{m_i}$ for distinct prime $p_i$, then
\[
d(n)=\prod_{i=1}^j d\left(p_i^{m_i}\right)\le C^{2^k}\prod_{i=1}^j \left(p_i^{m_i}\right)^\frac{1}{k}\le C^{2^k}n^\frac{1}{k}=O\left(n^\frac{1}{k}\right).
\]
\end{proof}

\begin{lemma} \label{fracsum}
Fix positive integers $k, m, a, b$. Then for positive integers $n$, the number of pairs of positive integers $(x,y)$ such that $\frac{m}{n}=\frac{a}{x}+\frac{b}{y}$ and all three fractions are in lowest terms is at most $O(n^\frac{1}{k})$.
\end{lemma}

\begin{proof}
Assume $\frac{m}{n}=\frac{a}{x}+\frac{b}{y}$. Let $p=\gcd(n,x)$, with $n=tp$ and $x=wp$. Then 
\[
\frac{b}{y}=\frac{m}{n}-\frac{a}{x}=\frac{mw-at}{twp}.
\]
Letting $q=\gcd(mw-at,twp)$, we get 
\begin{equation}\label{eq:getw}
mw-at=qb.
\end{equation}
For any choice of $n,p,q$, \eqref{eq:getw} gives at most one possible value of $w$, thus at most one value of $x$, and thus at most one value of $(x,y)$.

The definition of $q$ gives us $q\mid p$. Then for a given $n$, both $p$ and $q$ are divisors of $n$, so by Lemma~\ref{divisors} there are $O(n^{\frac{1}{2k}})$ possible values for $p$ and $O(n^{\frac{1}{2k}})$ values for $q$, so there are $O(n^{\frac{1}{k}})$ values for $(p,q)$ and $O(n^{\frac{1}{k}})$ pairs of numbers $(x,y)$.

\end{proof}

\section{Proof of Theorem~\ref{generalcase}}
Assume that $|A|=n$, $d^s_k(A)\ge\binom{n-1}{k}$, and that $A$ is not a $(k,s)$-anti-pencil. Note that then some $B\ni a_n$ has $\sum B\le\frac{s}{s+1}$, so since $1<k$, we have $a_n<\frac{s}{s+1}$. We will use this in all the cases below.  Also, the number of $k$-subsets of $A$ that are not $s$-divisors is at most $\binom{n}{k}-\binom{n}{k-1}=\binom{n-1}{k-1}$.

\begin{remark}\label{chains}
If $B$ and $C$ are $k$-subsets of $A$ with $B<C$, then $\sum B<\sum C$. Note that if $B_0<B_1<\cdots<B_m$ are all divisors of $A$ and $\sum B_m<s/q$, then $\sum B_0<s/(q+m)$. Therefore if $a\in B_0$, then $a<s/(q+m)$. Since $k<n$, $\sum B_m<s/s$, so we automatically get that $a<s/(s+m)$ 
\end{remark}

Each of the subsections below is a separate case.

\subsection{\texorpdfstring{$k$}{k} small}
Fix $2\le k$ and let $n>>k$.

For $1\le i_1,\ldots,i_k\le n$, call the ordered $k$-tuple $(i_1,\ldots,i_k)$ \emph{repetitive} if not all entries are distinct. Call it \emph{good} if all entries are distinct and $\{a_{i_j}\}$ is an $s$-divisor. Otherwise, call the ordered $k$-tuple \emph{bad}.

We will first restrict our attention to $k$-tuples where $i_k\ge n-1$. Among these, $O(n^{k-2})$ are repetitive. Also, $O(n^{k-2})$ include both $n$ and $n-1$ among their components. Of the remainder, at most $(k-1)!\binom{n-1}{k-1}\le n^{k-1}$ are bad. Thus at least 1/3 of the $k$-tuples $(i_1,\ldots,i_k)$ satisfying $i_k\ge n-1$ are good.


By the Pigeonhole Principle,
there are some values $j_2,\ldots,j_k$ with $j_k\ge n-1$ such that the chain $\{(1,j_2,\ldots,j_k),\ldots,(n,j_2,\ldots,j_k)\}\subset U$ has at least $n/3$ good $k$-tuples. This gives us a chain of $k$-subset $s$-divisors of length at least $n/3$. Thus $a_{n-1}\le\frac{3s}{n}$.

Let $B=\{a_i\mid i>\left(1-\frac{1}{9s^2}\right)n\}$. If $a_i\in B$, then
\[
1=\sum A=\sum_{i=1}^n a_i<na_i+\frac{n}{9s^2}a_{n-1}+a_n<na_i +\frac{1}{3s}+\frac{s}{s+1}
\]
so $na_i>\frac{1}{6s}$ and $a_i>\frac{1}{6sn}$.

Thus any $s$-divisors that is a subset of $B$ must sum to some $\frac{s}{m}>\frac{1}{6sn}$, so there are at most $6s^2n$ distinct values that $m$ can take. Thus there are at most $6s^2n$ distinct values that an $s$-divisor that is a subset of $B$ can sum to.

If $D\in\binom{B}{k-2}$ and $r=\frac{s}{m}$ for some positive integer $m$, call $D$ an \emph{r-stem} if there are at least $\frac{1}{10000s^6}n$ pairs $\{x,y\}\subset B\wo D$ with $\sum (D\cup\{x,y\})=r$. Call such pairs \emph{tails} of $D$. If two tails of $D$ are $\{x,y\}$ and $\{x,z\}$, then the sum condition gives us $y=z$, so tails of $D$ are pairwise disjoint.

Now let $B_0=B$. Note that $|B_0|>\frac{1}{10s^2}n$. As long as $|B_{i-1}|\ge\frac{1}{20s^2}n$, $B_{i-1}$ has at most $\binom{n-1}{k-1}$ subsets which are not $s$-divisors of $A$, so it has at least $\half\binom{|B_{i-1}|}{k}$ $k$-subsets that are $s$-divisors. Since these take on at most $6s^2n$ values, there must be some positive integer $m_i$ such that at least $\frac{1}{12s^2n}\binom{|B_{i-1}|}{k}$ $k$-subsets of $B_{i-1}$ sum to $r_i=\frac{s}{m_i}$.

If we randomly choose $D_i\in\binom{B_{i-1}}{k-2}$, the expected value for the number of pairs $\{x,y\}\subset B_{i-1}\wo D_i$ with $\sum(D_i\cup\{x,y\})=r_i$ is at least $\frac{1}{12s^2n}\binom{|B_{i-1}|-(k-2)}{2}$. Thus we will choose a $D_i$ such that the number of these pairs is at least $\frac{1}{12s^2n}\binom{|B_{i-1}|-(k-2)}{2}$.
Since  
\[
\frac{1}{12s^2n}\binom{|B_{i-1}|-(k-2)}{2}\ge \frac{1}{25s^2n}\left(|B_{i-1}|\right)^2\ge \frac{1}{25s^2n(20s^2)^2}n^2\ge \frac{1}{10000s^6}n,
\]
$D_i$ satisfies the definition of an $r_i$-stem.

Let $B_i=B_{i-1}\wo D_i$. Then for $i\le\frac{1}{20ks^2}n$, $D_i$ is an $r_i$-stem and all the $D_i$ are disjoint.

Since the number of $k$-subsets of $A$ which are not $s$-divisors is less than $\binom{\frac{1}{20ks^2}n}{k}$, we know that there must exist disjoint $D_{i_1},\ldots,D_{i_k}$ such that any set consisting of one element of each $D_{i_j}$ will be an $s$-divisor. Note that in the $k=2$ case, $D_{i_1}=D_{i_2}=\emptyset$. Partition $\bigcup_{j=1}^k D_{i_j}$ into $k-2$ such sets $C_1,\ldots,C_{k-2}$.

Let $p=\ceil{\frac{1}{10000s^6}n/(2k)}=\ceil{\frac{1}{20000s^6k}n}$. For $1\le j\le k$, we want to choose $T^j_1,\ldots,T^j_p$ to be tails of $D_{i_j}$. We will choose them for $j=1$, then for $j=2$, and so on. When we choose $\{T^j_\ell\}$, we will make each of these tails disjoint from each of the $k$ stems, as well as from the already chosen tails. This is possible since
\[
\left|\bigcup_{h=1}^k D_{i_h}\cup \bigcup_{h=1}^{j-1}  \bigcup_{\ell=1}^p T^h_\ell\right|=\left|\bigcup_{h=1}^k D_{i_h}\right|+\sum_{h=1}^{j-1} \left|\bigcup_{\ell=1}^p T^h_\ell\right|=k(k-2)+2(j-1)p\le k(k-2)+2(k-1)p.
\]
Since any element in a stem or in a previously chosen tail can be in at most one tail of $D_{i_j}$, at most $k(k-2)+2(k-1)p$ tails are eliminated, so there must be at least $p$ tails still available to choose from. 

We say that a choice of $k$ tails $\{T^j_{i_j}\}_{j=1}^k$ for each stem is \emph{fortuitous} if $\{x^j_{i_j}\}_{j=1}^k$ and $\{y^j_{i_j}\}_{j=1}^k$ are both $s$-divisors.
There are $p^k>n^k/(20000s^6k)^k$ choices of tails, and at most $\binom{n-1}{k-1}$ of them are not fortuitous.
Thus at least $\half$ of possible choices are fortuitous.

By the Pigeonhole Principle, we can choose  $i_1,\ldots,i_{k-1}$ so that there are at least $p/2$ choices for $i$ which make $\{T^1_{i_1},\ldots,T^{k-1}_{i_{k-1}},T^k_i\}$ fortuitous.

Note that different choices of $i$ give us different values of $x^k_i$ and therefore different values of $\sum_{j=1}^k x^j_{i_j}$, so $\sum_{j=1}^k x^j_{i_j}$ can take on at least
\[
p/2=\Omega(n)
\]
different values.

On the other hand, if we are given a fortuitous choice of tails $\{T^j_{i_j}\}$, then
\bal
\sum_{j=1}^{k-2}\sum C_j +\sum_{j=1}^k x^j_{i_j} +\sum_{j=1}^k y^j_{i_j}&=\sum_{j=1}^{k}\sum \left(D_{i_j}\cup\{x^j_{i_j},y^j_{i_j}\}\right)\\
\sum_{j=1}^k x^j_{i_j} +\sum_{j=1}^k y^j_{i_j}&=\sum_{j=1}^{k}r_{i_j}-\sum_{j=1}^{k-2}\sum C_j.
\eal
The right hand side does not depend on our choice of tails. Also, since each $r_{i_j}$ and each $\sum C_j$ has denominator at most $6s^2n$, the right hand side has denominator at most $(6s^2n)^{2k}$. Since both $\sum_{j=1}^k x^j_{i_j}$ and $\sum_{j=1}^k y^j_{i_j}$ are $s$-divisors, there are at most $s^2$ possibilities for their numerators. For each such possibility, by Lemma~\ref{fracsum}, $\sum_{j=1}^k x^j_{i_j}$ can take on at most
\[
O\left(\left(6s^2n)^{2k}\right)^\frac{1}{4k}\right)=O\left(sn^\half\right)
\]
different values. Thus $\sum_{j=1}^k x^j_{i_j}$ can take on at most
\[
O\left(s^3n^\half\right)
\]
different values, contradicting the upper bound above.

\subsection{\texorpdfstring{$n\ge\frac{3}{2}k$}{n>=3/2k}, \texorpdfstring{$k$}{k} sufficiently large}
Let $d=\ceil{\left(s(s+1)/0.03\right)^2}$.
Assume that $k$ is sufficiently large relative $d$ 
and that $n\ge\frac{3}{2}k$. 

Let $T_2$ be the set of $k$-subsets of $A$ that include both $a_{n-1}$ and $a_n$. Let $T_1$ be the set of $k$-subsets of $A$ that include one of $a_{n-1}$ or $a_n$, but not both. Define $U_1$ and $U_2$ similarly, but with $(k-d)$-subsets.

For $S\in U_t$, let $P_S=\{B\in T_t\mid S\subset B\}$ (the set of $k$-subsets obtainable by adding $d$ elements of $A$ less than $a_{n-1}$ to $S$).
Note that an element of $T_t$ is contained in $P_S$ for exactly $\binom{k-t}{d}$ values of $S$. 
Thus if $\alpha|T_t|$ elements of $T_t$ are $s$-divisors, then there is some $S\in U_t$ so that at least $\alpha|P_S|$ elements of $P_S$ are $s$-divisors.

Now note that the disjoint union $T_1\cup T_2$ is the set of all $k$-subsets whose greatest element is at least $a_{n-1}$, so
\[
|T_1\cup T_2|=\binom{n}{k}-\binom{n-2}{k}
\]
and the fraction of the elements of $T_1\cup T_2$ which are not $s$-divisors is at most
\bal
\frac{\binom{n}{k}-\binom{n-1}{k}}{\binom{n}{k}-\binom{n-2}{k}}&=\frac{1}{\frac{\binom{n}{k}-\binom{n-2}{k}}{\binom{n-1}{k-1}}}\\
                                                               &=\frac{1}{\frac{n}{k}-\frac{\binom{n-2}{k}}{\binom{n-1}{k-1}}}\\
                                                               &=\frac{1}{\frac{n(n-1)}{k(n-1)}-\frac{(n-k)(n-k-1)}{k(n-1)}}\\
																															 &=\frac{k(n-1)}{n(n-1)-(n-k)(n-k-1)}\\
																															 &=\frac{k(n-1)}{k(2n-1-k)}\\
																															 &=\frac{n-1}{2n-k-1}\\
																															 &\le 0.76
\eal
for sufficiently large $k$. Therefore, if we set $\alpha=0.24$, then for $t=1$ or $t=2$, the fraction of elements of $T_t$ that are $s$-divisors is at least $\alpha$, so for some $S$, the fraction of elements of $P_S$ which are $s$-divisors is at least $\alpha=0.24$.

Note that the partial order of $P_S$ is the same as the partial order of $\binom{A\wo S\wo\{a_{n-1},a_n\}}{d}$, so by Lemma~\ref{simplex}, its width is at most $\frac{2}{\sqrt{d}}\frac{1}{n-k-2}|P_S|$. Then, by Mirsky's theorem, there is a chain of $k$-subset $s$-divisors in $P_S$ of length at least
\[
\frac{\alpha|P_S|}{\frac{2}{\sqrt{d}}\frac{1}{n-k-2}|P_S|}=0.12\sqrt{d}(n-k-2)\ge (0.03\sqrt{d})n. 
\]
But then the first element of the chain includes $a_{n-1}$ or $a_n$, so by Remark~\ref{chains}, $a_{n-1}\le \frac{s}{(0.03\sqrt{d})n}$. Then
\[
\sum_{i=1}^{n-1} a_i\le \frac{s}{0.03\sqrt{d}}
\]
and, since $a_n<\frac{s}{s+1}$,
\[
\sum_{i=1}^n a_i<1
\]
yielding a contradiction.

\subsection{\texorpdfstring{$\frac{2}{3}n<k<n-\left(6s^2+3s\right)^2$}{2/3n<k<n-c}, \texorpdfstring{$k$}{k} sufficiently large}
Let $d=(6s^2+3s)^2$.
Assume that $k$ is sufficiently large and that $\frac{2}{3}n<k<n-d$.

Randomly arrange the elements of $A$ around a circle.
Let $M$ be the set of $k$-subsets of $A$ consisting of $k$ consecutive elements around the circle, and let $N=\{B\in M\mid \sum B\le\frac{1}{2(s+1)}\}$.
If $B, C\in N$ and they are shifted relative each other by at least $n-k-1$, then $|A\wo(B\cup C)|\le 1$, so 
\[\sum A\le\sum(A\wo(B\cup C))+\sum B+\sum C<\frac{s}{s+1}+\frac{1}{2(s+1)}+\frac{1}{2(s+1)}=1,
\]
which is impossible.

Thus any two elements of $N$ are shifted by at most $n-k-2$ around the circle. This gives us $|N|\le n-k-1$. Since any $k$-subset of $A$ summing to at most $\frac{1}{2(s+1)}$ has equal probability of being in $N$, this tells us that the number of $k$-subsets with sum at most $\frac{1}{2(s+1)}$ is at most $\frac{n-k-1}{n}\binom{n}{k}$. Thus there are at least
\[
\binom{n-1}{k}-\frac{n-k-1}{n}\binom{n}{k}=\frac{n-k}{n}\binom{n}{k}-\frac{n-k-1}{n}\binom{n}{k}=\frac{1}{n}\binom{n}{k}
\]
$k$-subsets which are $s$-divisors of $A$ and have a sum of elements greater than $\frac{1}{2(s+1)}$. The sum of elements of such a set is $\frac{s}{m}>\frac{1}{2(s+1)}$, so it can take on one of $2s(s+1)-s=2s^2+s$ values, so there must be some integer $m$ so that at least $\frac{1}{(2s^2+s)n}\binom{n}{k}$ of the $k$-subsets of $A$ sum to $\frac{s}{m}$. Thus at least $\frac{1}{(2s^2+s)n}\binom{n}{n-k}$ of the $(n-k)$-subsets of $A$ sum to $1-\frac{s}{m}$.

If $S\in\binom{A}{n-k-d}$, let $P_S$ be the set of $(n-k)$-subsets obtainable by adding $d$ elements of $A$ to $S$.
Note that any $(n-k)$-subset of $A$ is contained in $P_S$ for exactly $\binom{n-k}{d}$ values of $S$, so there is some $S$ so that at least
\[
\frac{1}{(2s^2+s)n}|P_S|
\]
elements of $P_S$ sum to $1-\frac{s}{m}$. 
They must then form an antichain.

However, the partial order of $P_S$ is the same as the partial order of $\binom{A\wo S}{d}$, so by Lemma~\ref{simplex}, its largest antichain has size less than
\[
\frac{2}{\sqrt{d}}\frac{1}{k+d}|P_S|<\frac{3}{n\sqrt{d}}|P_S|\le\frac{1}{(2s^2+s)n}|P_S|,
\]
yielding a contradiction.

\subsection{\texorpdfstring{$n-\left(6s^2+3s\right)^2\le k<n-1$}{n-C<=k<n-1}, \texorpdfstring{$k$}{k} sufficiently large}
Assume that $n-\left(6s^2+3s\right)^2\le k<n-1$. Let $u=n-k$. Thus $1<u\le \left(6s^2+3s\right)^2$, so $u$ can take on only finitely many values. Assume that $k$ is sufficiently large relative those values. Let
\[
Y=\left\{B\in\binom{A}{u}\;\middle|\; A\wo B\text{ is an }s\text{-divisor of } A\right\}.
\]
By assumption, $|Y|\ge\binom{n-1}{k}=\binom{n-1}{u-1}$.

Let $q$ be as small as possible so that $a_{n-q}<\frac{1}{u(s+1)}$. Note that $q<u(s+1)$. If $B\in\binom{A}{u}$ and $b\le a_{n-q}$ for all $b\in B$, then $\sum B<\frac{1}{s+1}$, so $\sum (A\wo B)>\frac{s}{s+1}$ and $B\notin Y$. Thus every $B\in Y$ contains at least one of the $q$ greatest elements of $A$.

The number of $u$-subsets of $A$ containing at least 2 of the $q$ greatest elements of $A$ is bounded by
\[
2^q\binom{n-q}{u-2}< 2^{u(s+1)}\binom{n}{u-2}<\half|Y|,
\]
so at least half of the elements of $Y$ contain exactly one of the $q$ greatest elements of $A$. 

Thus there must be some $a_i$ which is one of the $q$ greatest elements of $A$ such that at least $\frac{1}{2u(s+1)}\binom{n-1}{u-1}$ elements of $Y$ include $a_i$ and no other of the $q$ largest elements.

If $B$ is such an element of $Y$, then
\[
\sum B<a_i+(u-1)\frac{1}{u(s+1)}<\frac{s}{s+1}+\frac{u-1}{u(s+1)}=1-\frac{1}{u(s+1)}.
\]
Since $\sum B$ must be of the form $1-\frac{s}{m}$ for some positive integer $m$, we get fewer than $s(s+1)u$ possible values of $m$. Thus there must be some value of $m$ so that there are at least
\[
\frac{1}{2u^2s(s+1)^2}\binom{n-1}{u-1}
\]
different $u$-subsets of $A$ which include $a_i$ and sum to $1-\frac{s}{m}$. However, if we have a collection of that many $u$-subsets of $A$ that contain $a_i$, then 
some 2 of them will share $u-1$ elements and thus have different sum. This gives us a contradiction.

\section{Conclusion}
For $k$ sufficiently large, all $n$ are covered by one of the three last cases. For $k$ small, all but finitely values of $n$ are covered by the first case.

In the statement of Theorem~\ref{s1case} and Theorem~\ref{generalcase}, ``all but finitely many'' cannot be omitted. For example, Huynh\cite{huynh14} notes that $n=4, k=2$, $A=\{\frac{1}{24},\frac{5}{24},\frac{7}{24},\frac{11}{24}\}$ gives $d_k(A)=4>\binom{n-1}{k}$. As $s$ increases, the number of such exceptions grows; in fact, it is easy to see that any $n, k, A$, will be an exception for sufficiently large $s$.

We could follow the proof and trace out the upper bounds on $n$ such that $(k,n)$ is an exception; however these will probably be far from optimal (for instance, for $s=1$, $(2,4)$ is likely the only exception). It would be interesting to get a good bound on the number of such exceptions, or on how large $n$ can be in terms of $s$ for $(k,n)$ to be an exception.

In this paper, we are counting $B\in\binom{A}{k}$ such that $\sum B=\frac{s}{m}$. If we instead counted $B$ such that $\sum B<\frac{k}{n}$, this problem becomes equivalent to the Manickam-Mikl\'{o}s-Singhi conjecture:
\begin{conjecture}
For positive integers $n, k$ with $n\ge 4k$, every set of $n$ real numbers with nonnegative sum has at least $\binom{n-1}{k-1}$ $k$-element subsets whose sum is also nonnegative.
\end{conjecture}
The equivalence is given by taking the complement of $B$ and applying a linear transformation.

The MMS conjecture has been proven for $k\mid n$ \cite{ManickamMiklos1988}, $n\ge 10^{46}k$ \cite{Pokrovskiy2013}, and $n\ge 8k^2$ \cite{ChowdhurySarkisShariari2013}, however there are pairs $(n,k)$ such that it does not hold. This suggests a more general problem.

\begin{problem}
Fix $S\subseteq [0,1]$ and positive integers $n$ and $k$. If $A$ is a set of positive reals, let $d_k(S,A)$ be the number of subsets $B\in\binom{A}{k}$ such that $\sum B=S$. Let $d(S,k,n)$ be the maximal value of $d(S,k,n)$ over all $A$ with $|A|=n$ and $\sum A=1$. For what $S,k,n$ do we get $d(S,k,n)=\binom{n-1}{k}$? Furthermore, when does $d_k(S,A)\ge\binom{n-1}{k}$ imply that $A$ is an $k$-anti-pencil?
\end{problem}
This paper addresses this problem for $S=\{\frac{s}{m}\mid m\in \Zp\}$, while the MMS conjecture deals with this problem for $S=(0,k/n)$. Another example of a set for which this problem might be interesting is a set of the form $S=(0,\alpha k/n)\cup \{\frac{s}{m}\mid m\in \Zp\}$, which combines the theorem of this paper with the MMS conjecture.
\section{Acknowledgements}
This research was conducted as part of the University of Minnesota Duluth REU program, supported by NSA grant H98230-13-1-0273 and NSF grant 1358659. I would like to thank Joe Gallian for his advice and support. I would also like to thank Adam Hesterberg and Timothy Chow for helpful discussions and suggestions. I would also like to thank Brian Scott for a useful answer on Math Stack Exchange (http://math.stackexchange.com/questions/299770/width-of-a-product-of-chains).
\bibliographystyle{plain}
\bibliography{zbarskybib}
\end{document}